\documentclass[12pt]{amsart}
\usepackage[margin=1in]{geometry}
\usepackage{mathptmx}
\usepackage{mathtools}
\usepackage{amsthm}
\usepackage{amssymb}
\usepackage[alphabetic]{amsrefs}
\usepackage{graphicx}
\usepackage{tikz}
	\usetikzlibrary{decorations.pathreplacing}
	\usetikzlibrary{decorations.pathmorphing}
	\usetikzlibrary{patterns}
\usepackage{caption}
\usepackage[new]{old-arrows}

\newcommand{\ra}{\rightarrow}
\newcommand{\pr}{\prime}

\newcommand{\R}{\mathbb{R}}
\newcommand{\Z}{\mathbb{Z}}

\newcommand{\abs}[1]{\left\lvert #1 \right\rvert}
\newcommand{\tld}[1]{\widetilde{#1}}

\DeclareMathOperator{\img}{im}

\newtheorem{thm}{Theorem}

\newtheorem{lemma}{Lemma}

\theoremstyle{definition}
\newtheorem{defin}{Definition}
\newtheorem*{defin*}{Definition}
\newtheorem{ques}{Question}
\theoremstyle{remark}

\begin{document}

\title{Shrinking Without Doing Much At All}

\author{Michael Freedman}
\address{\hskip-\parindent
	Michael Freedman \\
	Microsoft Research, Station Q, and Department of Mathematics \\
	University of California, Santa Barbara \\
	Santa Barbara, CA 93106
}

\author{Michael Starbird}
\address{\hskip-\parindent
	Michael Starbird \\
	Department of Mathematics \\
	The University of Texas at Austin \\
	Austin, TX 78712
}

\begin{abstract}
    In 1952 Bing astonished the mathematical world with his wild involution on $S^3$. It has been among the most seminal examples in topology. The example depends on finding shrinking homeomorphisms of Bing's decomposition of $S^3$ into points and arcs. If Bing's original homeomorphisms are varied, Bing's original wild involution changes by conjugation, which preserves some analytic properties \cite{fs22} while altering others. In 1988, Bing published a second paper ``Shrinking Without Lengthening,'' answering a question that one of the present authors posed to him in an effort to understand the geometry of the entire conjugacy class. In this paper we produce a counterintuitive construction, namely, a method to shrink the Bing decomposition doing almost nothing at all--neither lengthening much nor rotating much.
\end{abstract}

\maketitle

\section{Introduction}
Prior to 1952, decomposition space theory (DST) was primarily a tool for two-dimensional topology. Moore's theorem that (in modern language) every upper semicontinuous decomposition of the plane is shrinkable was its high tide mark. In 1952, Bing produced a startlingly novel sequence of shrinking homeomorphisms, and with it a \emph{wild} involution on the three sphere $S^3$. The involution is wild in that it cannot be made \emph{smooth} in any system of coordinates. This single example invigorated several decades of research in DST, beginning with a rich, colorful theory in three dimensions, where manifold factors were first discovered \cite{ar65}. Then the subject jumped into high dimensions, where in separate works, J.\ Cannon and R.\ Edwards showed that the double suspensions of homology spheres are homeomorphic to a standard sphere \cites{can78,edw78}, a project that culminated in fundamental results of F.\ Quinn and R.\ Edwards on manifold recognition (see \cites{qui82,dav86} for an introduction). The Bing decomposition and its close relatives were also fundamental to the proof of the four-dimensional Poincar\'{e} Conjecture \cite{freedman82}. A.\ Dranishnikov and collaborators constructed remarkable dimension raising quotient maps \cite{ds86}. Combining Bing-style DST and Quinn-style surgery, J.\ Bryant, S.\ Ferry, W.\ Mio, and S.\ Weinberger constructed the modern theory of ANR homology manifolds \cite{bfmw96}.

After this 45-year burst of activity, DST visibly slowed. It had simply been \emph{too} successful in solving its core problems. The present paper is part of a reconsideration of DST with analytical aspects in mind. An earlier paper \cite{fs22} answered a long-standing question about the analytical properties of the Bing involution. It turns out that it can never be made Lipschitz or even quasi-conformal. In fact, any topological conjugate of the Bing involution was shown to have, up to an annoying polylog factor, an exponential modulus of continuity. The estimate for this intrinsic modulus of continuity (imoc) requires thinking, not just about \emph{one} shrink of the Bing decompostion $\mathcal{D}$, but \emph{all} possible shrinks (since the different shrinks can  be thought of as conjugates of any single one).

During the proof, we mostly found technical arguments to confirm our beliefs, but in one case we found a counter-intuitive surprise. The ``surprise shrink'' is the subject of this paper. We present it for two reasons. First, because Bing’s decomposition $\mathcal{D}$ is the ur-example of DST, any new insight into what is or is not required to shrink it should be recorded. Second, it is the hope of the authors that the method presented here might be combined with \cite{fs22} to strengthen the main result of that paper and remove the annoying polylog factor.  At first this hope seems odd since this paper provides a novel \emph{shrinking} method and \cite{fs22} is, in a sense, a \emph{non-shrinking} result: $\mathcal{D}$ can only be shrunk by doing great violence to its $\Z_2$-reflection symmetry. But in the proof of \cite{fs22} the polylog originates from the possibility that the imagined adversarial shrinker at some point starts to ``delay'' by making only tiny motions. The present paper gives some insight into what classes of such ``tiny motions,'' indeed, result in shrinking, and when they do, how to quantify the violations of symmetry. The reader is invited to join us in this game, of proving a strictly exponential imoc for the Bing involution; we have no proof, merely a hunch.

Beyond reverence for the Bing involution, what is the purpose of joining DST to analysis? Our answer is 4-manifolds. By an historical accident the topological theory of 4-manifolds arose simultaneously with Donaldson's theory of smooth 4-manifolds. Donaldson theory immediately implied that the infinite constructions of DST could not generally yield smoothable results. In a sense, the topologists were given an easy way out – a crisp no-go theorem. In 1982, there was no appetite to dig into shrinking arguments and determine exactly where, and how much, regularity was lost. The companion paper \cite{fs22} is a proof of principle that this work can be done. In dimension 3, all homeomorphisms are approximable by diffeomorphisms, so to frame the question, some additional structure, such as an involution, must be present. In dimension 4, the loss of smoothness through the infinite processes of DST has yet to be investigated geometrically, but perhaps is now in range.

The shrinking strategy presented here is an homage to Bing's final paper, “Shrinking Without Lengthening” \cite{bing88}, which Bing wrote to answer a question one of us asked him at the time. To integrate decomposition theory into analysis it is crucial to understand not just some shrinks, but all shrinks, of $\mathcal{D}$.

\emph{Bing's decomposition}. Bing's decomposition is made by intersecting finite stages called ``Bing rings,'' nested solid tori (see Figure \ref{fig:double}).  The shrink amounts to figuring out a strategy for stretching, twisting, bending, and/or rotating each finer pair of \emph{daughter} solid tori within the previous \emph{mother} stage. Bing found that the rings need be lengthened only infinitesimally during the shrink, hence his title. We find that not only can the length of the solid tori be nearly preserved, but the ``rotations'' of the daughters within the mother can also be made arbitrarily small and can be made to decay towards zero.

The Alexander Horned Sphere is the first known wild embedding of $S^2$ in $S^3$. One method of creating that wild embedding of $S^2$ is as follows: Start with a standard $S^2$ in $S^3$. For ease of visualization, think of $S^2$ as the $y\-z$ plane in $R^3$ with a point at infinity that makes $R^3$ into $S^3$ and makes the $y\-z$ plane a standard embedding of $S^2$ in $S^3$. On the positive $x$ side of this $S^2$, construct a specific Cantor set's worth of arcs--one end of each arc will be on $S^2$, while the other ends of those arcs will be rather entangled among each other. Specifically, the Cantor set's worth of arcs are created by taking the components of an infinite intersection of families $\{C_i\}_{i=0}^\infty$ of U-shaped solid cylinders. Each family $C_i$ has $2^i$ components. $C_0$ consists of a single U-shaped cylinder as pictured in Figure \ref{fig:decompositions}. Every component of $C_i$ contains two U-shaped cylinders of $C_{i+1}$ embedded as shown in Figure \ref{fig:decompositions}. Notice that the feet of each cylinder in $C_i$ get increasingly close to one another as $i$ increases. So in the limit, each component of $\cap_{i=0}^\infty C_i$ is an arc meeting $S^2$ in a single point, and the totality of those arc-endpoints is a Cantor set on $S^2$. Shrinking that Cantor set's worth of arcs is key to creating the Alexander Horned Sphere embedding of $S^2$. The set consisting of those arcs together with the remaining points of $S^3$ is an upper semi-continuous decomposition $\mathcal{D_A}$ of $S^3$. 

\begin{figure}[ht]
	\centering
	\begin{tikzpicture}
        \node at (0,0) {\includegraphics[scale=0.65]{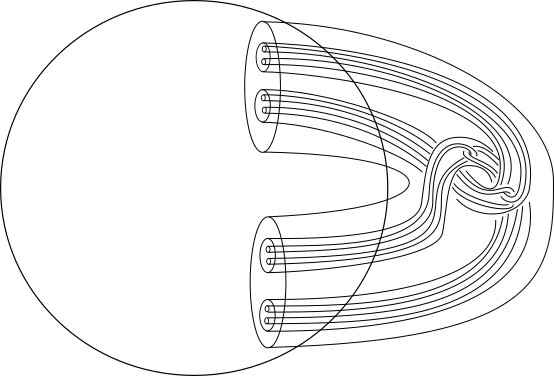}};
        \node at (-3.2,-2.2) {$S^2$};
        \node at (3.3,2.3) {$C_0$};
        \node at (5.35,0) {$C_1$};
        \draw[<->] (4.5,0.3) -- (5.05,0) -- (4.5,-0.3);
        \node at (1.35,0) {$C_2$};
        \draw[<->] (0.5,2) -- (1.05,0) -- (0.6,1.15);
        \draw[<->] (0.5,-1.95) -- (1.05,0) -- (0.6,-0.95);
    \end{tikzpicture}
	\caption{$C_2$ has 4 components, continuing the pattern as shown.}\label{fig:decompositions}
\end{figure}
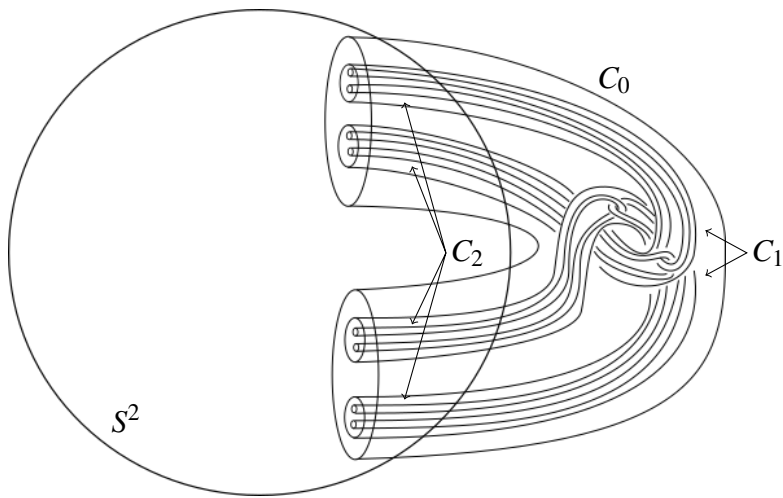

The Alexander Horned Sphere can be created as the image of $S^2$ under a surjective map $g^\pr:S^3 \ra S^3$ whose point pre-images are exactly the elements of $\mathcal{D_A}$. The function $g^\pr$ can easily be constructed as the limit of homeomorphisms $\{g_i^\pr:S^3 \ra S^3\}_{i=0}^\infty$ that increasingly compress the arcs of $\mathcal{D_A}$ toward their endpoints off $S^2$. Describing this process in terms of quotient spaces, the function $g^\pr$ shows that $S^3 \slash \mathcal{D}_A \cong S^3$, so the 3-sphere has not been changed, but the 2-sphere $g^\pr(S^2)$ is the \emph{wild} Alexander Horned Sphere. Its wildness is reflected in the fact that the right side component of its complement is no longer simply connected but instead has an infinitely generated fundamental group. The simple closed curve around the center of the cylinder $C_0$ is an example of a non-trivial element of the fundamental group of the wild component of $S^3 - g^\pr(S^2)$. The image under $g^\pr$ of the arcs in $\mathcal{D_A}$ is a wild Cantor set that lies on the Alexander Horned Sphere.

Consider the two \emph{closed} complementary regions of $S^3 - g^\pr(S^2)$. The closed left side region is clearly homeomorphic to the ball $B^3$, whereas the closed right side region, the Alexander Horned Ball, is clearly not, having a non-simply connected interior. 

Now we create a new decomposition $\mathcal{D}$ of $S^3$ by creating non-trivial elements that are symmetric across $S^2$. Let $I_0$ be the standard involution of $S^3$ with fixed point set $S^2$, that is, $I_0(x,y,z)=(-x,y,z)$. The non-degenerate elements of $\mathcal{D}$ are the components of the infinite intersection $\bigcap_{i=0}^\infty \tld{\mathcal{T}}_i$, where each $\tld{\mathcal{T}}_i = C_i \cup I_0(C_i)$. So each $\tld{\mathcal{T}}_i$ is the union of $2^i$ tori, each torus of which meets $S^2$ in two meridional disks. We will denote each torus component of $\tld{\mathcal{T}}_i$ by $\tld{T}_\sigma$ where $\sigma$ is a binary string of length $i$ and $\tld{T}_{\sigma 0}$ and $\tld{T}_{\sigma 1}$ are the two component tori of $\tld{\mathcal{T}}_{i+1}$ contained in $\tld{T}_\sigma$. $\tld{\mathcal{T}}_0$, $\tld{\mathcal{T}}_1$, and $\tld{\mathcal{T}}_2$ are drawn in Figure \ref{fig:double}. Notice that $\bigcap_{i=0}^\infty \tld{\mathcal{T}}_i$ is a Cantor set's worth of arcs, each one piercing $S^2$ at its center point. Those arcs comprise the non-degenerate elements of $\mathcal{D}$.

\begin{figure}[ht]
	\centering
	\input{Inserts/Fig2.tex}
	\caption{}\label{fig:double}
\end{figure}

In his 1952 paper, Bing showed that there is a surjective map $g:S^3 \ra S^3$ whose point pre-images are precisely the sets in $\mathcal{D}$, thereby producing an involution $I:S^3 \ra S^3$ defined by $I(x) = gI_0g^{-1}(x)$. The involution $I$ is wild since its fixed point set is the wild 2-sphere $g(S^2)$, and the closure of each component of the complement of $g(S^2)$ is an Alexander Horned Ball. The involution $I$ swaps sides across a wild sphere. 

Bing's map $g:S^3 \ra S^3$ is produced as the limit of homeomorphisms $\{g_i:S^3 \ra S^3\}_{i=0}^{\infty}$. The $g_i$ homeomorphisms eventually shrink the non-degenerate elements of $\mathcal{D}$ to increasingly smaller diameters as $i$ increases in such a way that the limit $g$ has the property that each non-degenerate element of $\mathcal{D}$ shrinks to a point and $\{g^{-1}(x) | x \in S^3\} = \mathcal{D}$. So the challenge is to produce the homeomorphisms $g_i$ that shrink the arcs of $\mathcal{D}$.

\emph{Bing's shrinks}.
The astonishing conclusion of \cite{bing52}
is that such shrinking homeomorphisms exist. Let us review two shrinks of $\mathcal{D}$ that Bing published in \cite{bing52} and \cite{bing88}
. In both shrinks, Bing reduces the shrinks to an essentially 1D model where the only important measure of diameter is displacement along the $x$-axis. Denoting the ``Bing tori'' dyadically, he lays them out along the $x$-axis and measures their diameters discretely by choosing a large integer $n$ and erecting parallel planes in intervals of $\frac{1}{n}$. So, the original tori are positioned as in Figure \ref{fig:plane_positions}.

\begin{figure}[ht]
	\centering
	\input{Inserts/Fig3.tex}
	\caption{}\label{fig:plane_positions}
\end{figure}

Bing's original shrink in his 1952 paper was accomplished by describing a sequence of homeomorphisms $\{g_i\}_{i=0}^\infty$. Each homeomorphism $g_{i+1}$ agrees with $g_i$ on $S^3 - \tld{\mathcal{T}_i}$. The homeomorphisms $g_i$ are defined in sets, meaning we first define the first $n_1$ $g_i$'s and then sort of pause while we celebrate a certain amount of shrinking, for example, we could choose our first collection of homeomorphisms such that for every component torus $\tld{T}_\sigma$ of $\tld{\mathcal{T}}_{n_1}$, diam($g_{n_1}(\tld{T}_\sigma)$) meets some diameter goal, say $< \frac{2}{10^1}$. Then we start with a new diameter goal, say, $< \frac{2}{10^2}$ and create the next set of $g_i$'s, say $i=n_1+1,...,n_2$. It is only at the end of each set of homeomorphisms that the diameters shrink. In other words, when we look at the diameters of the components of $\{g_j(\tld{\mathcal{T}}_j)\}_{j=n_1+1}^{n_2-1}$, generally those images of tori do not have increasingly smaller diameters; however, finally at stage $n_2$, for each torus $\tld{T}_\sigma \subset \tld{\mathcal{T}}_{n_2}$, diam$(g_{n_2}(\tld{T}_\sigma)) < \frac{2}{10^2}$.  

The homeomorphisms $\{g_i\}_{i=1}^{n_1}$ are defined as follows: let's choose our first goal to be to shrink each torus at some stage to diameter less than, say, $\frac{2}{10^1}$. Imagine the first stage torus as long, thin, and flat. Then position sufficiently many parallel planes, say $n_1$ planes, such that the distance between consecutive parallel planes is less than $\frac{1}{10^1}$ and such that the torus $\tld{T}_0$ intersects each of the $n_1$ planes in a pair of meridional disks. The first homeomorphism $g_1$ leaves $S^3 - \tld{\mathcal{T}}_0$ fixed and rotates $\tld{T}_0 \cup \tld{T}_1$ in $\tld{\mathcal{T}}_0$ such that each intersects one fewer plane (see Figure \ref{fig:rotated_rings}). For simplicity, for each component torus $\tld{T}_\sigma \subset \tld{\mathcal{T}}_i$, we will denote $g_i(\tld{T}_\sigma)$ by $T_\sigma$. So after the rotation and using our new notation, $T_0$ and $T_1$ each intersects only $n_1-1$ planes. In general, each subsequent pair of daughter tori are rotated in their already moved mother so that each of the $k$-stage daughters meet only $n_1-k$ planes (Figure \ref{fig:rotated_rings}).

\begin{figure}[ht]
	\centering
	\input{Inserts/Fig4.tex}
	\caption{}\label{fig:rotated_rings}
\end{figure}

After $n_1$ generations, no $T_\sigma$, $\sigma$ a binary word of length $\abs{\sigma} = n_1$, meets more than one of the planes, so its $x$-axis extent is $< \frac{2}{10^1}$. Normal to the $x$-axis, we are free to have chosen a strong compression, so this procedure produces a homeomorphism $g_{n_1}$ that shrinks each $n_1$-stage torus to diameter less than $\frac{2}{10}$. Shrinking the tori of course shrinks the decomposition elements therein. 

Next we choose a new diameter goal, say, $\frac{2}{10^2}$. Make many tick marks (actually, meridional disks on parallel planes) along the partially shrunk $T_\sigma$'s, $|\sigma|=n_1$ such that the distance between consecutive (around $T_\sigma$) meridional disks chosen is less than $\frac{1}{10^2}$. Now start our process over. That is, let $g_{n_1}+1$ rotate the two daughters in each $T_\sigma$ in such a way that those daughters each intersect one fewer meridional disk. Continue defining the $g_i$'s, each reducing the number of meridional disks intersected by each stage torus, until we reach a number $n_2$ such that every $T_\sigma$ where $|\sigma|=n_2$ meets at most one meridional disk. Again by compressing dimensions other than the $x$-extent means every such $T_\sigma$ has diameter less than $\frac{2}{10^2}$, as desired. Notice that, because of the folded nature of the $T_\sigma$'s where $|\sigma|=n_1$, the images under the $g_i$'s starting with $i=n_1+1$ do not decrease the diameters of the $T_\sigma$'s for a long time, but when we reach $n_2$, we can again pause to celebrate successful shrinking.  

So after sufficient celebration, we start again and repeat the process with an even more ambitiously small diameter goal. Continue producing such $g_i$'s. In the limit, the $g_i$'s converge to a surjective function $g:S^3 \ra S^3$ whose non-degenerate point pre-images are precisely the non-degenerate elements of $\mathcal{D}$. This then was Bing's original method of shrinking the decomposition $\mathcal{D}$.

Next we summarize Bing's 1988 shrink, which he produced in answer to questions we asked him at that time. In his 1988 shrink, every other rotation of tori is \emph{greedy}, as it tries (usually in vain) to cut diameters in half. The alternate rotations are \emph{patient}. It turns out greed does not speed the shrinking; it is only at steps indexed by $2^n - 1$, that diameters \emph{actually} are halved. So, again, many steps are taken during which no diameter shrinking is accomplished. In pictures here is the idea of Bing's 1988 shrink \cite{bing88} (Figure \ref{fig:halving}).

\begin{figure}[ht]
	\centering
	\begin{tikzpicture}
	    \node at (0,0) {\includegraphics[scale=0.9]{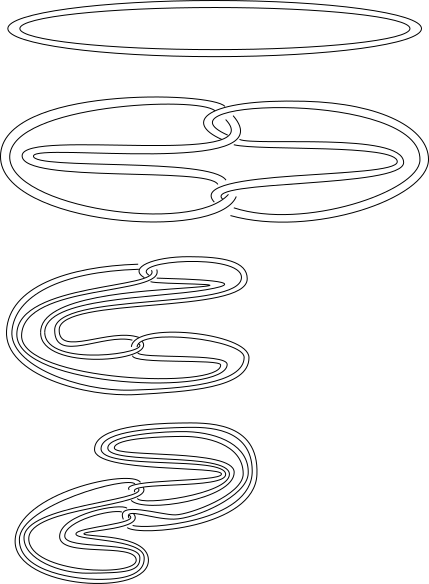}};
        \node at (1.7,5.1) {$\downarrow$ Step 1. Diam halves};
        \node at (0.5,1.2) {$\swarrow$ Step 2. Diam unchanged};
        \node at (3.2,1.2) {$\searrow$};
        \node at (3.7,0.9) {$\ddots$};
        \node at (0.1,-2.8) {$\searrow$ Step 3. Diam halves};
        \node at (-3.8,-2.8) {$\swarrow$};
        \node[rotate=90] at (-4.35,-3.3) {$\ddots$};
        \node at (0.5,-7.9) {$\downarrow$ Step $2^n-1$. Diam halves};
        \node at (-1.67,-7.3) {$\vdots$};
	\end{tikzpicture}
	\caption{}\label{fig:halving}
\end{figure}

\section{A Small Displacement Bing Shrink}
As described in the introduction, every known shrink of the Bing decomposition consists of starting with a standard torus and describing at each stage how to displace the clasp points of the two daughters relative to the two clasp points of the mother. For known shrinks, the clasps can be imagined as arbitrarily tight, and the reduction of diameter can be studied in a strictly 1D model where the starting torus is configured to tightly surround the unit interval. In this way, the diameter of a torus at a future stage is measured by the $x$-axis width that that folded solid torus covers of $[0,1]$.

Describing the positions of the folded tori at each stage can therefore be captured as a binary tree of functions into $[0,1]$ as follows. We start with an example.

\begin{figure}[ht]
    \centering
    \input{Inserts/Fig6.tex}
    \caption{}\label{fig:example}
\end{figure}

The first, straight torus $T_\varnothing$ is modeled by the identity function $f_\varnothing: [0,1] \ra [0,1]$. In the example in Figure \ref{fig:example}, the clasp points of the daughters $T_0$ and $T_1$ are displaced by distance $\frac{1}{3}$ at the left end and distance $\frac{1}{4}$ at the right end.

The following two functions, $f_0$ and $f_1$, describe the configurations of $T_0$ and $T_1$. $f_0: [-\frac{1}{3},\frac{3}{4}] \ra [0, \frac{3}{4}]$ takes the interval $[-\frac{1}{3},0]$ backwards from $\frac{1}{3}$ to 0 and then proceeds forwards from 0 to $\frac{3}{4}$, thus modeling the shape of $T_0$ in $T$. Likewise, $f_1: [\frac{1}{3}, \frac{1}{4}] \ra [\frac{1}{3},\frac{5}{4}]$.

Next we describe how to inductively produce a binary tree of functions $f_\sigma$ into $[0,1]$ whose images re-trace the patterns of the folded tori created by displacing the clasping points of daughter pairs relative to clasping points of their mothers. We define $f_\varnothing$ as the identity map on $[0,1]$.

Suppose $\sigma$ is a finite, binary string and the function $f_\sigma: [c_\sigma, d_\sigma] \ra [0,1]$ has been produced. The string $\sigma$ has two daughters $\sigma 0$ and $\sigma 1$, and $f_\sigma$ will have two daughters $f_{\sigma 0}$ and $f_{\sigma 1}$. The passage from $f_\sigma$ to its daughters depends on the choice of a pair of distinct points $a_\sigma, b_\sigma$ in $[c_\sigma,d_\sigma]$, $c_\sigma \leq a_\sigma \leq b_\sigma \leq d_\sigma$. We call $a_\sigma$ and $b_\sigma$ the daughter clasp points. Technically, each is a pair of nearby turning points, one for each daughter, but we abuse notation by denoting each pair of turning points as a single clasp point. No harm will result. The rule is that $f_{\sigma0}$ re-traces $f_\sigma$ backwards from $a_\sigma$ to $c_\sigma$, then forwards from $c_\sigma$ to $b_\sigma$, and $f_{\sigma1}$ re-traces $f_\sigma$ forward from $b_\sigma$ to $d_\sigma$ and then backwards from $d_\sigma$ to $a_\sigma$.

\begin{figure}[ht]
    \centering
    \begin{tikzpicture}
    \draw (-6,0) -- (6,0);
    \draw[fill=black] (-6,0) circle (0.25ex);
    \draw[fill=black] (-3.75,0) circle (0.25ex);
    \draw[fill=black] (-1.5,0) circle (0.25ex);
    \draw[fill=black] (2.5,0) circle (0.25ex);
    \draw[fill=black] (4.25,0) circle (0.25ex);
    \draw[fill=black] (6,0) circle (0.25ex);
    
    \node at (-6,-0.3) {$2c-a$};
    \node at (-3.75,-0.3) {$c$};
    \node at (-1.5,-0.3) {$a$};
    \node at (2.5,-0.3) {$b$};
    \node at (4.25,-0.3) {$d$};
    \node at (6,-0.3) {$2d-b$};
    \draw[->] (2,-1) -- (4.05,-1) arc (-90:90:0.2) -- (-1.9,-0.6);
    \draw[->] (-2,-0.6) -- (-3.55,-0.6) arc (90:270:0.2) -- (1.9,-1);
    \node at (-1.5,-1.3) {$f_{\sigma0}$};
    \node at (3,-1.3) {$f_{\sigma1}$};
    
    \node at (-1.25,-2.5) {$f_{\sigma0}: [2c-a,b] \ra [0,1]$};
    \node at (0,-3.1) {$f_{\sigma0}(x) = f_\sigma(-x + 2c),\ 2c-a \leq x \leq c$};
    \node at (-0.6,-3.7) {$f_{\sigma0}(x) = f_\sigma(x),\ c \leq x \leq b$, and};
    \node at (-1.25,-4.3) {$f_{\sigma1}: [a,2d-b] \ra [0,1]$};
    \node at (-1.05,-4.9) {$f_{\sigma1}(x) = f_\sigma(x),\ a \leq x \leq d$};
    \node at (0,-5.5) {$f_{\sigma1}(x) = f_\sigma(-x + 2d),\ d \leq x \leq 2d-b$};
\end{tikzpicture}
    \caption{}\label{fig:fmaps}
\end{figure}

In Figure \ref{fig:fmaps} all occurrences of $a$, $b$, $c$, and $d$ above implicitly carry a $\sigma$ subscript, dropped for readability.

Abstracting the concept of shrinking a decomposition, specifically Bing's decomposition $\mathcal{D}$, there are many interesting choices for the tree of \emph{clasp points} $\{a_\sigma,b_\sigma\}$, and hence the tree of functions $\{f_\sigma\}$.

\begin{defin}
    (1) A binary tree of functions $\{f_\sigma\}$ starting with $f_\varnothing$ being the identity on $[0,1]$ and defined as above will be called a \emph{Bing tree of functions}. (2) A Bing tree of functions $\{f_\sigma\}$ \emph{shrinks} iff $\operatorname{diam}(\img(f_\sigma)) \ra 0$ whenever the bit string length $\abs{\sigma} \ra \infty$.
\end{defin}

Let's refer to $\abs{a_\sigma - c_\sigma}$ and $\abs{d_\sigma - b_\sigma}$ as displacements. If the two displacements are equal, we can say the daughters are \emph{rotated} in the mother. Earlier, we put rotation in quotes because, in our construction, the displacements will not be exactly equal. The previously known shrinks of the Bing decomposition all included some large displacements, that is, instances where $\abs{a_\sigma - c_\sigma}$ and $\abs{d_\sigma - b_\sigma}$ were relatively large compared to $\abs{d_\sigma - c_\sigma}$. However, we show in this paper that it is possible to construct a Bing shrink, or equivalently, a Bing tree of functions $\{f_\sigma\}$ that shrinks even though all the displacements are small. The shrink we will produce has the additional property that the length of each torus, or equivalently the domain of each $f_\sigma$, grows by less than any desired quantity.

Before constructing our small displacement shrink, let's make some observations about the functions $f_\sigma$ in a Bing tree of functions. When visualizing the following, it might be useful to imagine the displacements, that is, the $\abs{a_\sigma - c_\sigma}$'s and $\abs{d_\sigma - b_\sigma}$'s as very small compared to the length of $\abs{d_\sigma - c_\sigma}$. Note that each function $f_\sigma: [c_\sigma, d_\sigma] \ra [0,1]$ is piecewise linear with each piece having slope $\pm 1$.

For specificity, we will discuss the daughter $f_{\sigma 0}$ of $f_\sigma$, the case for $f_{\sigma 1}$ being similar. The domain of $f_\sigma$, namely, $[c_\sigma, d_\sigma]$ shifts downward to create the domain of $f_{\sigma 0}$, namely, $[c_\sigma - (a_\sigma - c_\sigma) = 2c_\sigma - a_\sigma, b_\sigma]$. Let's think about the relationship between the images of $f_\sigma$ and $f_{\sigma 0}$.

First notice that the domains of $f_\sigma$ and $f_{\sigma 0}$ share the interval $[c_\sigma, b_\sigma]$, and, therefore, agree there. The function $f_{\sigma 0}$ is defined on the additional interval $[2 c_\sigma - a_\sigma, c_\sigma]$, but $f_{\sigma 0}([2c_\sigma - a_\sigma, c_\sigma])$ $= f_\sigma([c_\sigma, a_\sigma])$ as sets, so no new points are added to the image of $f_{\sigma 0}$ compared to $f_\sigma$. So $\img(f_{\sigma 0}) \subset \img(f_\sigma)$. These observations prove the following lemma that records how an $f_\sigma$ relates to its ancestors.

\begin{lemma}\label{lm:bin_string}
    Let the binary string $\tau$ be an ancestor of $\tau^\pr$ in a Bing tree of functions. Let $M[\tau, \tau^\pr] = \max\{a_\mu \mid \tau \leq \mu \leq \tau^\pr \}$ and $m_{[\tau, \tau^\pr]} = \min\{b_\mu \mid \tau \leq \mu \leq \tau^\pr\}$. Then
    \begin{enumerate}
        \item $f_{\tau^\pr}$ restricted to $[M_{[\tau, \tau^\pr]}, m_{[\tau, \tau^\pr]}] = f_\tau$ restricted to $[M_{[\tau, \tau^\pr]}, m_{[\tau, \tau^\pr]}]$

        \item $\operatorname{length}(f_{\tau^\pr}([c_{\tau^\pr}, M_{[\tau, \tau^\pr]}])) \leq \max\{\abs{a_\mu - c_\mu} \mid \tau \leq \mu \leq \tau^\pr \}$

        \item $\operatorname{length}(f_{\tau^\pr}[m_{[\tau,\sigma]}, d_{\tau^\pr}]) \leq \max\{\abs{d_\mu - b_\mu} \mid \tau \leq \mu \leq \tau^\pr\}$,
    \end{enumerate}
    where $\operatorname{length}()$ denotes the length of a subset of $[0,1]$.
\end{lemma}

During the coming construction of the $f_\sigma$'s, from time to time we will pause at an $f_\tau$ and then construct descendants in a prescribed manner until we pause again at an $f_{\tau^\pr}$. The pause positions will have the property that either $c_{\tau^\pr} = M_{[\tau,\tau^\pr]}$ or $d_{\tau^\pr} = m_{[\tau,\tau^\pr]}$. In such a case, $f_{\tau^\pr}$ consists of two parts---one part \emph{re-traces} $f_\tau$, namely (1) $f_{\tau^\pr}$ restricted to $[M_{[\tau, \tau^\pr]}, m_{[\tau, \tau^\pr]}] = f_\tau$ restricted to $[M_{[\tau,\tau^\pr]}, m_{[\tau, \tau^\pr]}]$, but if $c_{\tau^\pr} = M_{[\tau,\tau^\pr]}$ or $d_{\tau^\pr} = m_{[\tau,\tau^\pr]}$, then either the $M_{[\tau,\tau^\pr]}=c_{\tau^\pr}$ or $m_{[\tau,\tau^\pr]}=d_{\tau^\pr}$, and (2) a highly folded mapping with image $f_{\tau^\pr}([c_{\tau^\pr}, M_{[\tau, \tau^\pr]}])$ or $f_{\tau^\pr}([m_{[\tau, \tau^\pr]}, d_{\tau^\pr}])$. By the Lemma, the length of the images of the folds is less than $\epsilon_{D_{[\tau, \tau^\pr]}} = $ the maximum displacement among $\mu$'s between $\tau$ and $\tau^\pr$, which is either the right hand side of (2) or (3) in Lemma \ref{lm:bin_string}. We will call such an $f_{\tau^\pr}$ a $[A,\epsilon$-$W]$ function, where $A$ is the part of the domain on which $f_{\tau^\pr}$ agrees with or re-traces $f_{\tau}$ and $\epsilon$ is a bound on the length of the image of the remainder of the domain---the 'wiggles', hence the use of the letter '$W$'. Notice that in the domain of an $[A,\epsilon$-$W]$ function, the $A$ part could either be an interval at the lower end of the domain, as suggested by the notation, or the higher end of the domain interval.

\begin{figure}[ht]
    \centering
    \begin{tikzpicture}[scale=1.4]
    \draw (0,3) -- (0,-0.2);
    \draw (-0.2,0) -- (5,0);
    \node at (0,-0.4) {$c_\tau$};
    \draw (3.75,0.2) -- (3.75,-0.2);
    \node at (3.75,-0.4) {$d_\tau$};
    \draw (0,1) -- (0.5,0.5) -- (1.5,1.5) -- (2,1) -- (2.5,1.5) -- (2.75,1.25) -- (3.25,1.75) -- (3.75,1.25);
    
    \draw (0,-1) -- (0,-4.2);
    \draw (-0.2,-4) -- (5,-4);
    \node at (0,-4.4) {$c_\tau$};
    \draw (1.25,-3.8) -- (1.25,-4.2);
    \node at (1.25,-4.4) {$c_{\tau^\pr}$};
    \draw (3.75,-3.8) -- (3.75,-4.2);
    \node at (3.75,-4.4) {$d_\tau$};
    \draw (4.5,-3.8) -- (4.5,-4.2);
    \node at (4.5,-4.4) {$d_{\tau^\pr}$};
    \draw (1.25,-2.75) -- (1.5,-2.5) -- (2,-3) -- (2.5,-2.5) -- (2.75,-2.75) -- (3.25,-2.25) -- (3.75,-2.75) -- (3.8,-2.7) -- (3.85,-2.75) -- (3.9,-2.7) -- (3.95,-2.75) -- (4,-2.7) -- (4.05,-2.75) -- (4.1,-2.7) -- (4.15,-2.75) -- (4.2,-2.7) -- (4.25,-2.75) -- (4.3,-2.7) -- (4.35,-2.75) -- (4.4,-2.7) -- (4.45,-2.75) -- (4.5,-2.7);
    \node at (4.5,-2.75) {$]$};
    \node at (5.2,-2.8) {$< \epsilon_{D_{\tau, \tau^\pr}}$};
\end{tikzpicture}
    \caption{A $[A,\epsilon$-$W]$ function. Note that $d_{\tau^\pr}$ may be less than $d_\tau$, and that the part with $\epsilon$-bounded image wiggles could be at the lower end.}
\end{figure}
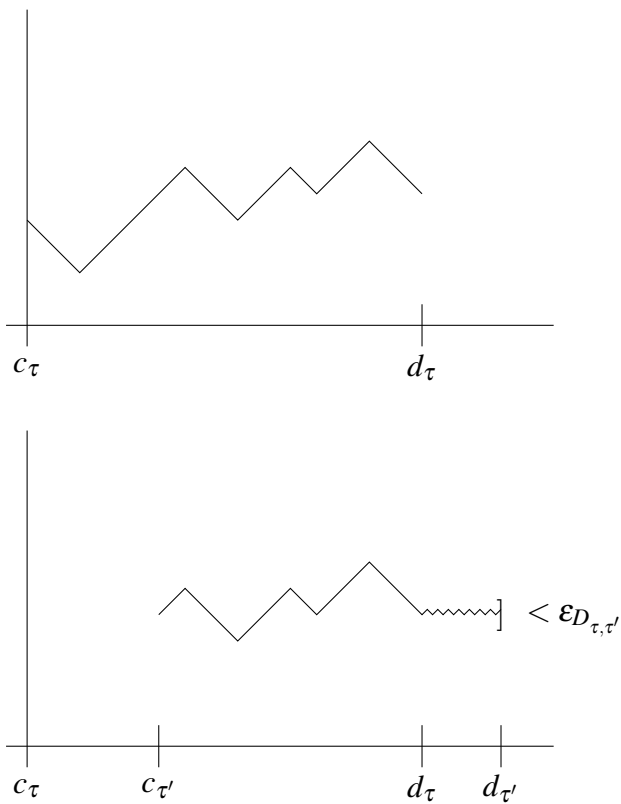

So now we are ready to construct a shrink that uses only small displacements.

\begin{thm}
    For any $\epsilon_L > 0$ and $\{\epsilon_i > 0 \mid \sum_{i=1}^\infty \epsilon_i^2 = \infty\}$, there exists a Bing tree of functions $\{f_\sigma\}$ that shrinks such that for every $\sigma$,
    \begin{enumerate}
        \item $\abs{d_\sigma - c_\sigma} < 1 + \epsilon_L$, and

        \item for $\abs{\sigma} = i$, $\abs{a_\sigma - c_\sigma} < \epsilon_i$ and $\abs{d_\sigma - b_\sigma} < \epsilon_i$.
    \end{enumerate}
\end{thm}

\begin{proof}

    We will describe the functions $f_\sigma$ by describing the domain intervals $[c_\sigma, d_\sigma$]. Notice that knowing the domain intervals of $f_\sigma$ and of $f_{\sigma 0}$ automatically implies what $a_\sigma$ and $b_\sigma$ are and what $f_{\sigma 1}$ is. Specifically, if the domain of $f_\sigma$ is $[c_\sigma, d_\sigma]$ and the domain of $f_{\sigma 0}$ is $[c_{\sigma 0}, d_{\sigma 0}]$, then $a_\sigma = c_\sigma + (c_\sigma - c_{\sigma 0}) = 2c_\sigma - c_{\sigma 0}$ and $b_\sigma = d_{\sigma 0}$. So the domain of $f_{\sigma 1}$ is also determined.

    It will be convenient to associate each domain interval $[c_\sigma, d_\sigma]$ with the point in the plane $(c_\sigma, d_\sigma)$. Every such point $(c_\sigma, d_\sigma)$ lies above the main diagonal $\Delta$ and the horizontal (or vertical) distance from $\Delta$ to $(c_\sigma, d_\sigma)$ equals the length of $[c_\sigma, d_\sigma]$. Notice that interval inclusion, that is, $[u, v] \subset [c_\sigma, d_\sigma]$, corresponds to the point $(c_\sigma, d_\sigma)$ lying in the NW quadrant with respect to $(u,v)$.

    Notice that the vector from the point in the plane $(c_\sigma, d_\sigma)$ to $(c_{\sigma 0}, d_{\sigma 0})$, $\overrightarrow{(c_\sigma - a_\sigma, b_\sigma - d_\sigma)}$, is the negative of the vector from $(c_\sigma, d_\sigma)$ to $(c_{\sigma 1}, d_{\sigma 1})$. These observations will allow us to construct the desired $f_\sigma$'s by describing a binary tree of points in the plane.

    Our goal is to create $(c_\sigma, d_\sigma)$'s with length $< 1 + \epsilon_L$ and $\operatorname{length}(f_\sigma([c_\sigma, d_\sigma]))$ getting short. Our iterative procedure always begins at a $[A,\epsilon$-$W]$ function. Suppose $(c_\tau, d_\tau)$ is a $[A,\epsilon$-$W]$ function. Then $\operatorname{length}(f_\tau(c_\tau, d_\tau))$ is less than the length of the domain of the re-trace part plus $\epsilon$, the maximum length of the image of the part with small wiggles.
    
    \emph{The strategy in brief}. (1) Sometimes we work on shortening the domain of the re-trace part while keeping the same $\epsilon$ of the wiggle part. Since all functions are piecewise linear with slopes of each piece $\pm 1$, the image of the re-trace part will always be less than or equal to the length of the domain of the re-trace part. (2) After shortening the re-trace part, the $\epsilon$ of the wiggle part limits how short the image is, hence the wiggle part needs attention. So we then choose a smaller $\epsilon$, say $\frac{\epsilon}{2}$, and begin again, that is, we think of the entire interval $[c_\tau, d_\tau]$ as the re-trace part and proceed to shorten the re-trace part using $\frac{\epsilon}{2}$-wiggles.

    This process of shortening the re-trace part domain---in a process that requires many stages---with the same size $\epsilon$-wiggles, followed by starting again with a smaller $\epsilon$ yields the result that the images of descendant $f_\sigma$'s eventually have increasingly smaller images.

    So now we need to describe the procedure for moving from point $(c_\sigma, d_\sigma)$ to points $(c_{\sigma 0}, d_{\sigma 0})$ and $(c_{\sigma 1}, d_{\sigma 1})$.

    \emph{Shrinking the re-trace part.} Suppose $f_\tau$ is a $[A,\epsilon$-$W]$ function. That is, $f_\tau: [c_\tau, d_\tau] \ra [0,1]$ is a piecewise linear function with each piece having slope $\pm 1$, $\abs{d_\tau - c_\tau} < 1 + \epsilon_L$ and there is a $t \in [c_\tau, d_\tau]$ such that (i) $\operatorname{length}(f_\tau([t,d_\tau])) < \epsilon$ or (ii) $\operatorname{length}(f([c_\tau, t])) < \epsilon$. Say (i) is the case (the (ii) case being similar), so $[c_\tau, t]$ is the re-trace part. Then there exists a finite Bing tree of functions starting with $f_\tau$ such that for every final descendant $f_{\tau^\pr}$, $f_{\tau^\pr}$ is a $[A,\epsilon$-$W]$ function such that $\abs{d_{\tau^\pr} - c_{\tau^\pr}} < 1 + \epsilon_L$ and there is a $t^\pr \in [c_{\tau^\pr}, d_{\tau^\pr}]$ such that either (i)$\abs{t^\pr - c_{\tau^\pr}} < \frac{2}{3} \abs{t - c_\tau}$ and  $\operatorname{length}(f_{\tau^\pr}([t^\pr,d_{\tau^\pr}])) < \epsilon$ or (ii) $\abs{t^\pr - d_{\tau^\pr}} < \frac{2}{3} \abs{t - c_\tau}$ and  $\operatorname{length}(f_{\tau^\pr}([c_{\tau^\pr}, t^\pr])) < \epsilon$. In other words, the length of the domain of the re-trace part of $f_{\tau^\pr}$ is less than $\frac{2}{3}$ the length of the domain of the re-trace part of $f_\tau$, while the image of the remainder of the domain has length less than $\epsilon$.

    Figure \ref{fig:NW_quad} shows the NW quadrant over the point corresponding to the middle third of the re-trace part domain of $f_\tau$. The re-trace part domain of $f_\tau$ is $[c_\tau, t]$. So, the middle third is $[c_\tau + \frac{1}{3}(t - c_\tau), c_\tau + \frac{2}{3}(t - c_\tau)]$.

    \begin{figure}
        \centering
        \begin{tikzpicture}[scale=1.4]
    \draw (0,0) -- (8,0);
    \draw (5,-1.5) -- (5,5);
    \draw (-1.1,0.05) arc (-76:-20:7.9);
    \draw (-0.7,0) arc (-74:-20:8.1);
    \draw (-0.3,-0.05) arc (-72:-20:8.3);
    \draw (0.1,-0.1) arc (-70:-20:8.5);
    \draw (0.5,-0.15) arc (-68:-20:8.7);
    \draw (0.9,-0.2) arc (-66:-20:8.9);
    \draw (1.3,-0.25) arc (-64:-20:9.1);
    \draw (1.7,-0.3) arc (-62:-20:9.3);
    \draw (2.1,-0.35) arc (-60:-20:9.5);
    \draw (2.5,-0.4) arc (-58:-20:9.7);
    \draw (2.9,-0.4) arc (-56:-20:9.9);
    \draw (3.3,-0.4) arc (-54:-20:10.1);
    \draw (3.7,-0.4) arc (-52:-20:10.3);
    \draw (4.1,-0.4) arc (-50:-20:10.5);
    \draw (4.5,-0.4) arc (-48:-20:10.7);
    \draw (4.9,-0.4) arc (-46:-20:10.9);
    
    \draw[fill=black] (3.32,3) circle (0.2ex);
    \node at (2.8,3.1) {$(c_\tau, d_\tau)$};
    \draw (2.2,1.53) -- (4.42,4.4);
    \draw[fill=black] (2.2,1.53) circle (0.2ex);
    \node at (1.5,1.6) {$(c_{\tau 0}, d_{\tau 0})$};
    \draw[fill=black] (4.42,4.4) circle (0.2ex);
    \node at (3.7, 4.5) {$(c_{\tau 1}, d_{\tau 1})$};
    \draw (0.85,0.43) -- (3.57,2.6);
    \draw (3.8,2.9) -- (4.73,5.1);
    \node at (6.6,-0.5) {$(c_\tau + \frac{1}{3}(t-c_\tau),$};
    \node at (7,-0.9) {$c_\tau + \frac{2}{3}(t - c_\tau)$};
    \draw[->] (5.5,-0.4) -- (5.1,-0.1);
    
    \node at (4.2,-2.2) {Blow-up view};
    \draw (2.9,-5.2) arc (-70:-20:3);
    \draw (3.2,-5.3) -- (4.9,-3.6);
    \draw (3,-5.38) arc (-70:-15:3.3);
    \draw (2.3,-5.68) arc (-85:-2:3.4);
    \draw (2.4,-5.67) -- (4.34,-4.77);
    \draw (4.5,-4.6) -- (5.3,-2.6);
    \draw[fill=black] (4,-4.5) circle (0.2ex);
    \draw[fill=black] (3.2,-5.3) circle (0.2ex);
    \draw[fill=black] (4.9,-3.6) circle (0.2ex);
    \draw[fill=black] (4.34,-4.77) circle (0.2ex);
    \draw[fill=black] (4.5,-4.6) circle (0.2ex);
    \node at (3.3,-4.3) {$(c_{\tau 0}, d_{\tau 0})$};
    \node at (2.1,-5.15) {$(c_{\tau 00}, d_{\tau 00})$};
    \node at (4.3,-3.2) {$(c_{\tau 01}, d_{\tau 01})$};
    \node at (5.2,-5) {$(c_{\tau 001}, d_{\tau 001})$};
    \node at (5.55,-4.6) {$(c_{\tau 010}, d_{\tau 010})$};
\end{tikzpicture}
        \caption{}\label{fig:NW_quad}
    \end{figure}

    The figure also suggests a large number of concentric circles centered at some distant point in that NW quadrant on the slope $-1$ ray from $(c_\tau + \frac{1}{3}(t - c_\tau), c_\tau + \frac{2}{3}(t - c_\tau))$ heading up and left. Recall that $\abs{d_\tau - c_\tau} < 1 + \epsilon_L$. The center point of the circles is so distant that every point $(x,y)$ on the arc of the circle containing $(c_\tau, d_\tau)$ in the pictured NW quadrant has $y - x < 1 + \epsilon_L$. This choice of circle center will guarantee that no point $(c_\sigma, d_\sigma)$ that is created during our process has length greater than $1+\epsilon_L$.

    The bullseye pattern of circles is chosen such that for any point on one of those circles, the distance along the tangent to the next circle is $< \min\{\epsilon, \epsilon_i\}$ where $\epsilon$ is the wiggle width $\epsilon$ and $i = \abs{\sigma}$ at each stage.

    Now we are ready to construct our binary tree of points $(c_\sigma, d_\sigma)$. We begin at $(c_\tau, d_\tau)$ and proceed along the tangent of its circle in both directions until we hit the next circle. Those two points will be $(c_{\tau 0}, d_{\tau 0})$ and $(c_{\tau 1}, d_{\tau 1})$. From each of those points we do the same thing---that is, from $(c_{\tau 0}, d_{\tau 0})$ we move along the tangent of its circle to find the points $(c_{\tau 00}, d_{\tau 00})$ and $(c_{\tau 01}, d_{\tau 01})$ on the next circle.

    We continue creating this binary tree of points until we arrive at a point $(c_{\tau_1}, d_{\tau_1})$ where $c_{\tau_1} > c_\tau + \frac{1}{3} (t-c_\tau)$ or $d_{\tau_1} < c_\tau + \frac{2}{3}(t - c_\tau)$, that is, when $(c_{\tau_1}, d_{\tau_1})$ does not lie in the NW quadrant with respect to $(c_\tau + \frac{1}{3}(t - c_\tau), c_\tau + \frac{2}{3}(t - c_\tau))$. At such a point we pause to state a simple lemma.
    
    \begin{lemma}
        Let $\{\epsilon_i > 0 \mid \sum_{i=0}^\infty \epsilon_i^2 = \infty\}$ and let $\{c_i\}_{i=1}^\infty$ be a nested sequence of concentric circles in $\R^2$ centered at point $c$ such that for every $i$, the radius of $C_i$ is $r_i$ and the distance from a point on $C_i$ along the tangent to $C_i$ to a point on $C_{i+1}$ is $\epsilon_i$. Then the sequence of radii $(r_i)_{i=1}^\infty$ is unbounded.
    \end{lemma}
    
    \begin{proof}
        Let $M \in \R^+$. Given the hypotheses we will show that there is a $k$ such that $r_k \geq M$. If not, then for every $i$, $r_i < M$.
        
        \begin{figure}[ht]
        \centering
        \begin{tikzpicture}
            \draw (0,0) circle (1.5);
            \draw (0,0) circle (3);
            \draw[fill=black] (0,0) circle (0.25ex);
            \draw (0,0) -- (0,-1.5) -- (2.6,-1.5) -- cycle;
            \node at (-0.4,3.2) {$C_{i+1}$};
            \node at (-0.3,1.7) {$C_i$};
            \node at (0,0.3) {$c$};
            \node at (-0.3,-0.75) {$r_i$};
            \node at (1.3,-1.8) {$\epsilon_i$};
            \node at (1.8,-0.7) {$r_{i+1}$};
        \end{tikzpicture}
        \end{figure}
        
        By the Pythagorean Theorem,
        \[
            r_{i+1}^2 - r_i^2 = \epsilon_i^2\text{, so } r_{i+1} - r_i = \frac{\epsilon_i^2}{r_{i+1} + r_i} > \frac{\epsilon_i^2}{2M}
        \]
        
        But $\sum \frac{\epsilon_i^2}{2M}$ diverges, meaning some $r_k > M$.
    \end{proof}

    Let us now notice two things: (1) Every path branch of the binary tree of points stops by a predictable stage since our hypothesis about the divergence of the sum of the squares of the $\epsilon_i$'s guarantees that there are only a finite number of circles before the point $(c_\tau + \frac{1}{3}(t - c_\tau), c_\tau + \frac{2}{3}(t - c_\tau))$ is itself inside the circle.  Any point constructed with our procedure on that circle must be outside the NW quadrant. (2) Since we stop at the first time that $c_{\tau_1} > c_\tau + \frac{1}{3}(t - c_\tau)$ or $d_{\tau_1} < c_{\tau_1} + \frac{2}{3}(t - c_\tau)$, say $c_\tau > c_\tau + \frac{1}{3}(t - c_\tau)$, then $c_{\tau_1}$ must equal $M_{[\tau, \tau_L]}$, so $f_{\tau_1}$ is a $[A,\epsilon$-$W]$ function. That is, $f_{\tau_1}$ agrees with $f_\tau$ on $[M_{[\tau, \tau_1]}, m_{[\tau, \tau_1]}] = [c_{\tau_1}, m_{[\tau, \tau_1]}]$ and $f_{\tau_1}([m_{[\tau, \tau_1]}], d_{\tau_1}])$ must have image length less than $\epsilon$. Therefore, the re-trace part of each such $f_{\tau_1}$ must have domain length less than $\frac{2}{3}$ the domain length of the re-trace part of $f_\tau$, while the image of the wiggly part has length less than $\epsilon$; therefore, $\operatorname{length}(f_{\tau_1}([c_{\tau_1},d_{\tau_1}])$ is less than $\frac{2}{3}$ the domain length of the re-trace part of $f_\tau$; plus $\epsilon$.  And the domain of $f_{\tau_1}$, that is, $[c_{\tau_1},d_{\tau_1}]$, has length less than $1 + \epsilon_L$.

    We can now repeat the process of shrinking the re-trace part using the same $\epsilon$ size wiggles until the re-trace part itself has domain length less than $\epsilon$. At that point, the total length of the image of that function must be less than $2\epsilon$ and its domain length is less than $1+\epsilon_L$. We now proceed to produce functions with images less than $\epsilon$ by thinking of the entire domain as the re-trace part, but now using $\frac{\epsilon}{2}$-wiggles. Then we repeat the entire process infinitely often producing functions with image lengths less than $\frac{\epsilon}{2}$, $\frac{\epsilon}{4}$, $\frac{\epsilon}{8}, \dots$, thereby creating the desired shrink.
\end{proof}
    
\section{Conclusion and Questions}
    
Bing's 1952 construction of a wild involution of $S^3$ opened the door to many further insights---and many further questions. In \cite{fs22}, we proved that all conjugates of Bing's involution \emph{must} share certain analytic features with the involution derived from Bing's original shrink. Understanding such shared features turned up a surprise, the subject of this paper, whose exact relation to the previous paper is still to be worked out. Bing always told us not to form fixed beliefs about what you do not know. Following his advice, let us state some questions without presuming to guess their answers.

\begin{ques}
    The first question involves the rate at which small displacements can shrink. Bing's original method of shrinking and his shrinking without lengthening method shrink stage tori to size about $\frac{1}{n}$ their original diameter at stage order $n$. The analysis of \cite{as89} shows that there is no faster way to shrink the Bing decomposition $\mathcal{D}$. It appears if one imposes constraints on displacement and lengthening, as we have here, shrinking must be even slower. For example, if lengthening and displacement are restricted to 0.1\%, our algorithm takes about $10^{12}$ stages to get from diameter $= 1$ to diameter $= 0.001$. What is the actual functional form for shrinking using our algorithm? And are there more efficient algorithms respecting the same constraints? In both cases, what are the analytical properties, the modulus of continuity (moc), of the corresponding involutions?
\end{ques}
    
\begin{ques}
    Suppose we use the 1D model for shrinking in this paper. We showed that it is possible to shrink the Bing decomposition with small displacement shrinks. Bing's shrinks and the shrink in this paper seem to require some insight or even cleverness, but might that apparent cleverness be an illusion? Suppose the $a_\sigma$'s and $b_\sigma$'s were simply chosen randomly in the intervals $[c_\sigma,d_\sigma]$. Would such a random selection lead to a shrink of the Bing decomposition with probability $1$? If so, wouldn't we feel silly. Shrinking (or not) is a \emph{tail event}, meaning independent of any initial segment of choices, so Kolmogorov's 0-1 law tells us that within a probabilistic model, shrinking will occur with either probability 0 or 1. Which, depends on the model. If the model is artificially concentrated near our explicit shrink, the probability will be 1, but if $\{(a_\sigma, b_\sigma)\}$ are independent and uniformly distributed, we do not know.
\end{ques}

We were slow to accept that Bing's decomposition could be shrunk using only tiny jiggles. Bing understood the unknown is actually \emph{unknown}. He told us that he would work from both directions, keep an open mind, and not care which way the truth turns out.

\bibliography{references}

\end{document}